\documentclass[reqno]{amsart}
\usepackage[all]{xy}
\usepackage{mathptmx}
\usepackage{amssymb}
\usepackage{amsfonts}
\usepackage{amsmath}
\usepackage{graphicx}

\newtheorem{thm}{Theorem}[section]
\newtheorem{corollary}[thm]{Corollary}
\newtheorem{lemma}[thm]{Lemma}
\newtheorem{proposition}[thm]{Proposition}

\theoremstyle{definition}
\newtheorem{definition}[thm]{Definition}
\newtheorem{example}[thm]{Example}

\catcode`\ç=13
\defç{\c{c}}
\catcode`\é=13
\defé{\'e}
\catcode`\à=13
\defà{\`a}
\catcode`\è=13
\defè{\`e}
\catcode`\â=13
\defâ{\^a}
\catcode`\ù=13
\defù{\`u}
\catcode`\ê=13
\defê{\^e}
\catcode`\î=13
\defî{\^\i}
\catcode`\ô=13
\defô{\^o}

\newcommand{\field}[1]{\mathbb{#1}}
\newcommand{\Z }{\field{Z}}

\def\pd{{\rm pd}}
\def\id{{\rm id}}
\def\fd{{\rm fd}}
\def\Gpd{{\rm Gpd}}
\def\Gfd{{\rm Gfd}}
\def\Gid{{\rm Gid}}

\def\lgldim{{l\rm.gldim}}
\def\rgldim{{r\rm.gldim}}

\def\lGgldim{{l\rm.Ggldim}}
\def\rGgldim{{r\rm.Ggldim}}

\def\gldim{{\rm gldim}}
\def\Ggldim{{\rm Ggldim}}

\def\wdim{{\rm wgldim}}

\def\Im{{\rm Im}}
\def\Ker{{\rm Ker}}
\def\Ext{{\rm Ext}}
\def\Tor{{\rm Tor}}
\def\Hom{{\rm Hom}}

\begin{document}

\title[$(n,m)$-SG rings]{$(n,m)$-SG rings}

\author{Driss Bennis}
\address{Department of Mathematics, Faculty of Science and
Technology of Fez,\\ Box 2202, University S. M. Ben Abdellah, Fez,
Morocco} \email{driss\_bennis@hotmail.com}







\begin{abstract} This paper is a continuation of the
paper Int. Electron. J. Algebra 6 (2009),  219--227. Namely, we
introduce and study a doubly filtered set of classes of  rings of
finite Gorenstein global dimension, which are called $(n,m)$-SG
for integers $n\geq 1$ and $m\geq 0$.  Examples of $(n,m)$-SG
rings, for $n=1\; and\; 2$ and every $m\geq 0$, are given.
\end{abstract}

\maketitle

\section{Introduction}

Throughout the paper all rings are associative with identity,  and
all modules are,
if not specified otherwise, left modules.\\
Let $R$ be a ring and let $M$ be an $R$-module.  For an
$R$-module $M$, we use $\pd_R(M),\ \id_R(M)$, and $\fd_R(M)$ to
denote, respectively, the classical projective, injective and flat
dimension of $M$.   We use $\lgldim(R)$ and $\rgldim(R)$ to
denote, respectively, the classical left and right global
dimension of $R$, and   $\wdim(R)$ to denote the weak global
dimension  of $R$.\bigskip

The Gorenstein homological dimensions theory originated in the
works of Auslander and Bridger  \cite{A1} and \cite{A2}, where
they introduced the G-dimension of any finitely generated module
$M$ and over any Noetherian ring $R$. The G-dimension is analogous
to the classical projective dimension and shares some of its
principal properties (see \cite{LW} for more details). However, to
complete the analogy an extension of the G-dimension to
non-necessarily finitely generated modules is needed. This is done
in \cite{GoIn, GoInPj}, where the Gorenstein projective dimension
was defined over arbitrary rings (as an extension of the
G-dimension to modules that are not necessarily finitely
generated), and the Gorenstein injective dimension was defined as
a dual notion of the Gorenstein projective dimension:

\begin{definition}\label{def1}
 Let $R$ be a ring. \begin{itemize}
    \item    An  $R$-module  $M$ is called \textit{Gorenstein projective}
(G-projective for short)  if there exists an exact sequence of
projective  $R$-modules,
$$\mathbf{P}=\ \cdots\rightarrow P_1\rightarrow P_0 \rightarrow
P_{ -1 }\rightarrow P_{-2 } \rightarrow\cdots,$$ such that  $M
\cong \Im(P_0 \rightarrow P_{ -1 })$ and such that $\Hom_R ( -, Q)
$ leaves the sequence $\mathbf{P}$ exact whenever $Q$ is a
projective  $R$-module.\\ The exact sequence $\mathbf{P}$ is
called a
\textit{complete projective resolution} of  $M$.\\
\indent For a positive integer $n$, we say that $M$ has
\textit{Gorenstein projective dimension} at most $n$, and we write
$\Gpd_R(M)\leq n$ (or simply $\Gpd(M)\leq n$), if there is an
exact sequence of  $R$-modules,
$$ 0 \rightarrow G_n\rightarrow \cdots \rightarrow G_0\rightarrow
M \rightarrow 0,$$ where each $G_i$ is G-projective.
\item  Dually, the \textit{Gorenstein injective} module
(G-injective for short) is defined, and so the \textit{Gorenstein
injective dimension}, $\Gid_R(M)\leq n$, of an $R$-module $M$ is
defined.\end{itemize}
\end{definition}

Also to complete the analogy with the classical homological
dimensions theory, the Gorenstein flat dimension was introduced in
\cite{GoPlat} as follows:

\begin{definition}\label{def2}
 Let $R$ be a ring. An $R$-module  $M$  is called \textit{Gorenstein flat} (G-flat for short)  if there
exists an exact sequence of flat   $R$-modules, $$ \mathbf{F}=\
 \cdots\rightarrow F_1\rightarrow F_0 \rightarrow F^0 \rightarrow
F^1 \rightarrow\cdots,$$ such that  $M \cong \Im(F_0 \rightarrow
F^0)$ and such that $I\otimes_R - $ leaves the sequence
$\mathbf{F}$  exact whenever $I$ is an injective right
$R$-module.\\ The exact sequence $\mathbf{F}$ is called a
\textit{complete flat resolution} of $M$.\\
\indent For a positive integer $n$, we say that $M$ has
\textit{Gorenstein flat dimension} at most $n$, and we write
$\Gfd_R(M)\leq n$, if there is an exact sequence of $R$-modules,
$$ 0 \rightarrow G_n\rightarrow \cdots \rightarrow G_0\rightarrow
M \rightarrow 0,$$ where each $G_i$ is G-flat.
\end{definition}

\indent The Gorenstein homological dimensions have been
extensively studied by many others, who proved that these
dimensions share  many nice properties of the classical
homological dimensions (see for instance \cite{LW,CFoxH,
Rel-hom}). Recently, in  \cite{B5}, a  particular case of modules
of finite Gorenstein projective dimension  is introduced as
follows:

\begin{definition}\label{DefnSGproj} Let $R$ be a ring and let $n\geq 1$ and $m\geq 0$ be
integers. An  $R$-module $M$ is called \textit{$(n,m)$-strongly
Gorenstein projective}    ($(n,m)$-SG-projective for short) if
there exists an exact sequence of $R$-modules, $$0\rightarrow
M\rightarrow Q_n\rightarrow\cdots\rightarrow Q_1 \rightarrow
M\rightarrow 0,$$ where $\pd_R(Q_i)\leq m$  for $1\leq i\leq n $,
such that $ \Ext^i_R(M,Q ) = 0 $ for every $i> m$   and  every
projective $R$-module $Q$.
\end{definition}

The $(1,0)$-SG-projective modules are already investigated in
\cite{BM}  (see \cite[Proposition 2.9]{BM}). They are called
strongly Gorenstein projective modules (SG-projective modules for
short) (see also \cite{LY} and \cite{GZ}). In  \cite{BM2},
$(n,0)$-SG-projective modules are first studied, and they are
called $n$-strongly Gorenstein projective modules
($n$-SG-projective modules for short). In general,
$(n,m)$-SG-projective modules are a particular case of modules
with Gorenstein projective dimension at most $m$ \cite[Theorem
2.4]{B5}. The $(1,m)$-SG-projective modules are served to
characterize  modules of Gorenstein projective dimension at most
$m$ in a similar way to the way SG-projective modules characterize
G-projective modules (see \cite[Corollary 2.8]{B5} and
\cite[Theorem 2.7]{BM}). Namely, we have that a module  $M$  has
Gorenstein projective dimension at most  a positive integer $m$ if
and only if $M$ is a direct summand of  a $(1,m)$-SG-projective
module. As mentioned at the end of the paper \cite{B5}, dually the
$(n,m)$-SG-injective modules  are defined.\bigskip

In this paper, we continue the investigation of
$(n,m)$-SG-projective and  $(n,m)$-SG-injective modules. Namely,
we are interested in studying rings over which all modules are
$(n,m)$-SG-projective (resp., $(n,m)$-SG-injective). First, we
show, for a ring $R$, that the assertions ``all $R$-modules are
$(n,m)$-SG-projective" and ``all $R$-modules are
$(n,m)$-SG-injective" are equivalent  (Proposition
\ref{prop-main}). A ring that satisfies one of these equivalent
assertions is called left $(n,m)$-SG (Definition \ref{Def-nmSG}).
In the main result of this paper (Theorem \ref{thm-main}), left
$(n,m)$-SG rings are characterized in terms of left Gorenstein
global dimension: the left Gorenstein global dimension of a ring
$R$, $\lgldim(R)$, is defined in \cite{BM3} as  the common value
of the equal quantities \cite[Theorem 1.1]{BM3}:
  $$\sup\{\Gpd_R(M)\,|\,M\;\mathrm{is\;an}\;R\!-\!\mathrm{module}\}  =
\sup\{\Gid_R(M)\,|\,M\;\mathrm{is\;an}\;R\!-\!\mathrm{module}\}.$$
Namely, after giving  a characterization of $(n,m)$-SG-projective
(resp., $(n,m)$-SG-injective) modules (Lemmas \ref{lem-chara-GP}
and \ref{lem-chara-GI}), we show, in Theorem \ref{thm-main},  that
the $(n,m)$-SG rings are particular cases of rings with left
Gorenstein global dimension at most $m$. Then, for Noetherian
rings, $(n,m)$-SG rings are particular cases of $m$-Gorenstein
rings (see Theorem \ref{thm-Noe} and Corollary \ref{cor-nm-Noe}).
So the $(n,0)$-SG rings are   particular examples of the
well-known quasi-Frobenius rings (see   Corollary
\ref{cor-n0-Noe}). In particular, $(1,0)$-SG commutative rings are
investigated in \cite{BMO}, and they are called SG-semisimple. It
is proved that a local $(1,0)$-SG commutative ring is just a ring
with only one non-trivial ideal \cite[Theorem 3.7]{BMO}.\bigskip

After investigating some relationships between $(n,m)$-SG rings
(Proposition \ref{prop-relations}),   the remain of the paper is
devoted to establish examples of $(n,m)$-SG rings. For that, we
study  the notion of $(n,m)$-SG rings in direct product of rings,
such that we prove (Proposition \ref{prop-prod}):  \begin{center}
A  direct product of rings $\displaystyle\prod_{i=1}^n R_i$  is
left $(n,m)$-SG  if and only if each $R_i$ is left
$(n,m)$-SG.\end{center}
 Then, a family of left $(1,i)$-SG rings
which are not left $(1,i-1)$-SG, for every $i\geq 1$, are given
(Example \ref{exm-1}). In Proposition \ref{prop-2m}, we show that
if $R$ is a commutative ring with $\gldim(R )=m$ for an integer
$m\geq 1$, such that $R$ contains a non-zero divisor element $x$,
then the quotient ring $ R /xR $ is $(2,m-1)$-SG. As a
consequence, we give examples of left $(2,0)$-SG rings which are
not left $(1,0)$-SG; and, for $i\geq 1$, we construct examples of
left $(2,i)$-SG rings which are neither  left $(1,i )$-SG nor
$(2,i-1)$-SG.

\section{Main results}

We start with the following result:

\begin{proposition}\label{prop-main}  Let $n\geq 1$ and $m\geq 0$ be
integers. For  a   ring  $R$,  the following assertions are
equivalent:
\begin{enumerate}
    \item Every $R$-module is $(n,m)$-SG-projective;
    \item Every  $R$-module is $(n,m)$-SG-injective.
\end{enumerate}
\end{proposition}
\begin{proof} We prove only the implication $(1) \Rightarrow (2)$. The
implication $(2) \Rightarrow (1)$ has a dual proof.\\
 First, using \cite[Proposition 2.3]{HH}, note that
every $R$-module $P$ with finite projective dimension has
injective dimension at most $m$. Indeed, $ \Ext^i(N,P) = 0 $ for
any $i> m$ and every $R$-module $N$ (since, by hypothesis, $N$ is
$(n,m)$-SG-projective). Also, note that every injective $R$-module
$I$ has projective dimension at most $m$. In fact, as
$(n,m)$-SG-projective, $I$ embeds in an $R$-module with projective
dimension  at most $m$, and since $I$ is injective it is a direct summand of a such $R$-module.\\
Now, consider an $R$-module $M$. Since $M$ is
$(n,m)$-SG-projective, there exists an exact sequence of
$R$-modules,
$$0\rightarrow M\rightarrow Q_n\rightarrow\cdots\rightarrow Q_1
\rightarrow M\rightarrow 0,$$ where $\pd_R(Q_i)\leq m$. By the
reason above $\id_R(Q_i)\leq m$; and also $ \Ext^i_R(J,M ) = 0 $
for any $i> m$ and   every injective $R$-module $J$ (since, by the
reason above, $\pd_R(J )\leq m$). Therefore, $M$ is
$(n,m)$-SG-injective.\end{proof}

\begin{definition}\label{Def-nmSG} A ring $R$ is called
left (resp., right) $(n,m)$-SG  for some integers  $n\geq 1$ and
$m\geq 0$, if $R$ satisfies one of the equivalent conditions of
Proposition  \ref{prop-main} for left  (resp., right) $R$-modules.
If $R$ is both left and right  $(n,m)$-SG, we simply say  that it
is $(n,m)$-SG.
\end{definition}

Later, we give examples of $(n,m)$-SG rings. Now, we set the main
result of this paper which gives a characterization of $(n,m)$-SG
rings in terms of left Gorenstein global dimension (for a
background on left Gorenstein global dimension, see
\cite{BM6,BM3,BM4}). For that, we need the following key lemma
which gives a characterization of $(n,m)$-SG-projective modules
(see also its $(n,m)$-SG-injective version, Lemma
\ref{lem-chara-GI}). This result is a generalization of
\cite[Theorem 2.7]{B5} and so it gives an affirmative answer to
the question concerning the converse of \cite[Theorem 2.4]{B5}
(see the note before \cite[Lemma 2.5]{B5}).\bigskip

Recall, for a projective resolution of a module $M$, $$ \cdots
\rightarrow P_1 \rightarrow P_0\rightarrow M\rightarrow 0,$$ that
the module $K_i=\Im( P_i \rightarrow P_{i-1})$  for $i\geq 1$, is
called an $i^{th}$ syzygy of  $M$.

\begin{lemma} \label{lem-chara-GP} Let $R$ be a ring  and let $n\geq 1$ and $m\geq 0$ be
integers. For  an $R$-module $M$  the following assertions are
equivalent:
\begin{enumerate}
    \item $M$ is $(n,m)$-SG-projective;
    \item $ \Gpd_R(M) \leq m $ and an $m^{th}$ syzygy of   $M$ is $(n,0)$-SG-projective.
    \item There exists a short exact sequence of  $R$-modules, $0\rightarrow P  \rightarrow G \rightarrow   M
    \rightarrow  0$, where $G$ is $(n,0)$-SG-projective    and
    $\pd_R(P) \leq m-1$;
    \item There exists a short exact sequence of  $R$-modules, $0\rightarrow M \rightarrow Q
    \rightarrow H \rightarrow  0$, where $H$ is $(n,0)$-SG-projective and
    $\pd_R(Q) \leq m $.
\end{enumerate}
\end{lemma}
\begin{proof}    $(1)\Rightarrow (2).$ Follows from \cite[Theorem
2.4]{B5}.\\
$(2)\Rightarrow (3).$  A similar proof to the one of \cite[Theorem
2.4(3)]{B5} shows that  any $k^{th}$ syzygy of   $M$ is
$(n,0)$-SG-projective, where $k= \Gpd_R(M) \leq m $. Then, we have
an exact sequence  of $R$-modules,
$$0\rightarrow G_k\rightarrow P_{k-1}\rightarrow\cdots\rightarrow
P_0 \rightarrow M\rightarrow 0,$$ where $P_i $  are projective and
$G_k$ is $(n,0)$-SG-projective. Consider a right half of a
complete projective resolution of $G_k$, $$0\rightarrow
G_k\rightarrow Q_{k-1}\rightarrow\cdots\rightarrow Q_0 \rightarrow
N\rightarrow 0,$$  where $Q_i $  are projective and, by
\cite[Lemma 1.2(2)]{B5},  $N$ is $(n,0)$-SG-projective. Then, from
\cite[Proposition 1.8]{HH}, we get the following commutative
diagram:
$$\begin{array}{ccccccccccccc}
   0&\rightarrow &G_k&\rightarrow &Q_{k-1}&\rightarrow&\cdots&\rightarrow &Q_0& \rightarrow &N&\rightarrow &0 \\
    &  &\parallel&  &\downarrow& & &  &\downarrow&   &\downarrow&  & \\
   0&\rightarrow & G_k&\rightarrow &P_{k-1} &\rightarrow&\cdots&\rightarrow &  P_0& \rightarrow &M&\rightarrow &0 \\
\end{array}$$
This diagram gives a chain map between complexes,
$$\begin{array}{ccccccccccc}
   0 & \rightarrow &Q_{k-1}&\rightarrow&\cdots&\rightarrow &Q_0& \rightarrow &N&\rightarrow &0 \\
    &  &   \downarrow& & &  &\downarrow&   &\downarrow&  & \\
   0&   \rightarrow &P_{k-1} &\rightarrow&\cdots&\rightarrow &  P_0& \rightarrow &M&\rightarrow &0 \\
\end{array}$$
which induces an isomorphism in homology. Then, its mapping cone
is exact (see \cite[Section 1.5]{Wei}). That is, the following
exact sequence:
$$0\rightarrow Q_{k-1} \rightarrow P_{k-1 }\oplus Q_{k-2} \rightarrow\cdots
\rightarrow P_{1}\oplus Q_{0}\rightarrow  P_{0}\oplus N\rightarrow
M\rightarrow 0 $$ Therefore, The sequence, $0\rightarrow P
\rightarrow G \rightarrow   M
    \rightarrow  0$, where $G=P_{0}\oplus N$ and $P=\Ker( P_{0}\oplus N\rightarrow
M)$, is the desired sequence.\\
$(3)\Rightarrow (4).$ Since  $G$ is $(n,0)$-SG-projective, there
exists a short exact sequence of  $R$-modules, $0\rightarrow G
\rightarrow F
    \rightarrow H \rightarrow  0$, where $F $ is projective and $H$ is
    $(n,0)$-SG-projective. Then, with the sequence $0\rightarrow P  \rightarrow G \rightarrow   M
    \rightarrow  0$ we get the following pushout  diagram:
    $$ \xymatrix{
     &  0 \ar[d]  & 0 \ar[d]  &   &  \\
 &  P\ar[d] \ar@{=}[r]&  P\ar[d]  &   &  \\
 0\ar[r]& G \ar[d] \ar[r] & F\ar@{-->}[d] \ar[r] & H\ar@{=}[d]  \ar[r] & 0\\
0\ar[r]&M \ar@{-->}[r] \ar[d]& Q\ar[d] \ar[r]& H\ar[r] & 0\\
 & 0 &0  &   & }$$
From the middle exact sequence, $\pd_R(Q)=\pd_R(P)+1\leq m$.
Therefore, the bottom sequence is the desired short exact
sequence.\\
$(4)\Rightarrow (1).$ First, using  the short exact sequence
$0\rightarrow M \rightarrow Q        \rightarrow H \rightarrow 0$,
one can show   that  $ \Ext^i_R(M,L ) = 0 $ for every $i> m$ and
every projective  $R$-module $L$. Then, it remains, by definition,
to prove the existence of an exact sequence of the form:
$$0\rightarrow M\rightarrow L_n\rightarrow\cdots\rightarrow L_1
\rightarrow M\rightarrow 0, $$ where every $L_i$ has projective
dimension at
most $m$.\\
Since $H$ is $(n,0)$-SG-projective, there exists an exact sequence
of modules:
$$0\rightarrow H\rightarrow Q_n\rightarrow\cdots\rightarrow Q_1
\rightarrow H\rightarrow 0 $$ Decomposing this   sequence into
short exact sequences
$$0\rightarrow H_{i+1}\rightarrow      Q_i\rightarrow
H_{i }\rightarrow 0,$$ where $H_{n+1}= H=H_1$. And consider the
following family of short exact sequences
$$(\alpha_i)\qquad 0\rightarrow G_{i } \rightarrow F_{i}\rightarrow
H_{i}\rightarrow 0,$$ where, for $i=2,...,n$,  $P_{i }$ is
projective  and $G_i$ is G-projective, and, for $i= 1\
\mathrm{and}\ n$,  the short exact sequence  $(\alpha_i)$ is the
sequence  $0\rightarrow M \rightarrow Q \rightarrow H \rightarrow
0$. Then, since $\Ext^1_R (H_i,L)=0 $ for every $R$-module $L$
with finite projective dimension \cite[Proposition 2.3]{HH}, we
get, from  the Horseshoe Lemma (a dual version of \cite[Lemma
1.7]{HH}), the following family of  commutative diagrams
$(\beta_i)$:
$$\begin{array}{cccccccccc}
   &  &0&  & 0&   &0&  &  \\
 &  &\downarrow&  &\downarrow&   &\downarrow&  & \\
   0&\rightarrow &H_{i+1}& \rightarrow & Q_i& \rightarrow   &H_i&\rightarrow &0 \\
 &  &\downarrow&  &\downarrow&   &\downarrow&  & \\
   0&\rightarrow &F_{i+1}& \rightarrow & F_{i+1}\oplus F_i& \rightarrow   &F_i&\rightarrow &0 \\
              &  &\downarrow&  &\downarrow&   &\downarrow&  & \\
 0&\rightarrow &G_{i+1}& \rightarrow & L_i& \rightarrow   &G_i&\rightarrow &0 \\
                &  &\downarrow&  &\downarrow&   &\downarrow&  & \\
               &  &0&  &0&   &0&  &
\end{array}$$
Then, we obtain a family of short exact sequences $(\theta_i)$:
$$\begin{array}{ll}
 (\theta_1):&\quad 0\rightarrow G_2\rightarrow L_1\rightarrow M \rightarrow 0,\\
 (\theta_i):&\quad 0\rightarrow  G_{i+1 } \rightarrow L_i\rightarrow G_i
\rightarrow
  0
,\ \mathrm{for} \ i=2,...,n-1\ \mathrm{and}  \\
 (\theta_n):&\quad0 \rightarrow  M\rightarrow L_n \rightarrow G_n\rightarrow 0 \\
\end{array}  $$
such that, from the middle sequences of the commutative diagrams
$(\beta_i)$, $L_i$ is projective for $ i=2,...,n-1$,
$\pd_R(L_1)\leq m$, and $\pd_R(L_n)\leq m$. Therefore, we get the
desired sequence by assembling the short exact sequences
$(\theta_i)$.\end{proof}

Also, one can prove the following dual version of Lemma
\ref{lem-chara-GP}. Recall, for an injective resolution of a
module $M$, $$0 \rightarrow M\rightarrow I_0\rightarrow  I_1
\rightarrow \cdots ,$$ that the module $K_i=\Im( P_{i-1}
\rightarrow P_{i })$  for $i\geq 1$, is called an $i^{th}$
cosyzygy of  $M$.

\begin{lemma} \label{lem-chara-GI} Let $R$ be a ring  and let $n\geq 1$ and $m\geq 0$ be
integers. For  an $R$-module $M$  the following assertions are
equivalent:
\begin{enumerate}
    \item $M$ is $(n,m)$-SG-injective;
    \item $ \Gid_R(M) \leq m $ and an $m^{th}$ cosyzygy of   $M$ is
    $(n,0)$-SG-injective;
    \item There exists a short exact sequence of  $R$-modules, $0\rightarrow M    \rightarrow G \rightarrow
    I   \rightarrow  0$, where $G$ is $(n,0)$-SG-injective   and
    $\id_R(I) \leq m-1$;
    \item There exists a short exact sequence of  $R$-modules, $0\rightarrow H \rightarrow
    J \rightarrow M\rightarrow  0$, where $H$ is $(n,0)$-SG-injective and
    $\id_R(J)  \leq m $.
\end{enumerate}
\end{lemma}

Now, we can prove our  main result:

\begin{thm}\label{thm-main}  Let $n\geq 1$ and $m\geq 0$ be
integers. For  a   ring  $R$,  the following assertions are
equivalent:
\begin{enumerate}
    \item   $R$ is $(n,m)$-SG;
    \item $ \Ggldim(R) \leq m $ and  every G-projective $R$-module is
    $(n,0)$-SG-projective;
    \item $ \Ggldim(R) \leq m $ and  every G-injective  $R$-module is
    $(n,0)$-SG-injective.
\end{enumerate}
\end{thm}
\begin{proof}  We prove only the equivalence $(1)\Leftrightarrow (2).$
The equivalence $(1)\Leftrightarrow (3)$ has a dual proof.\\
$(1)\Rightarrow (2).$ We have $ \Ggldim(R) \leq m $ since every
$(n,m)$-SG-projective  $R$-module has Gorenstein projective
dimension at most $m$ (by Lemma \ref{lem-chara-GP}$(1)
\Rightarrow(2)$). From \cite[Lemma 2.6(1)]{B5}, we get that every
G-projective  $R$-module is $(n,0)$-SG-projective.\\
 $(2)\Rightarrow (1).$  Let $M$ be an $R$-module. Since $ \Ggldim(R) \leq m
 $, every $m^{th}$ syzygy of  $M$ is  G-projective, which
 is, by hypothesis, $(n,0)$-SG-projective. Therefore, by Lemma \ref{lem-chara-GP}$(2)
\Rightarrow(1)$, $M$ is $(n,m)$-SG-projective.\end{proof}

As a consequence, the Noetherian $(n,m)$-SG  rings are particular
cases of $m$-Gorenstein rings:  a ring $R$ is said to be
$m$-Gorenstein, for a positive integer $m$, if it is left and
right Noetherian with self-injective dimension at most $ m$ on
both the left and the right sides \cite[Definitions 9.1.1 and
9.1.9]{Rel-hom}.  The $m$-Gorenstein rings are characterized in
terms of Gorenstein homological dimensions (see \cite[Theorem
12.3.1]{Rel-hom}) and in terms of classical homological dimensions
(see  \cite[Theorem 9.1.11]{Rel-hom}). In the following result, we
rewrite this list of properties that   characterize the
$m$-Gorenstein rings, and we enlarge it using the notions of
$(n,m)$-SG-projective and $(n,m)$-SG-injective modules.

\begin{thm}[\cite{Rel-hom}, Theorems 9.1.11 and 12.3.1]\label{thm-Noe}
If $R$ is  a left and right Noetherian ring, then, for a positive
integer $m$, the following are equivalent:
\begin{enumerate}
    \item $R$ is $m$-Gorenstein;
    \item $\lGgldim(R)\leq m$;
    \item $\rGgldim(R)\leq m$;
     \item $\id_R(M)\leq m$ for every projective left (resp., right) $R$-module $M$;
     \item $\pd_R(M)\leq m$ for every injective left (resp., right)
     $R$-module  $M$;
  \item $\Gid_R(M)\leq m$ for every G-projective left (resp., right) $R$-module  $M$;
     \item $\Gpd_R(M)\leq m$ for every G-injective left (resp., right)
     $R$-module  $M$;
      \item  For every integer $n\geq 1$, every    $(n,0)$-SG-projective left (resp., right) $R$-module is
     $(n,m)$-SG-injective;
    \item  For every integer $n\geq 1$, every   $(n,0)$-G-injective left (resp., right) $R$-module is
     $(n,m)$-SG-projective.
\end{enumerate}
\end{thm}
\begin{proof} The equivalences $(1) \Leftrightarrow (2) \Leftrightarrow
(3)$ follow from  \cite[Theorem  9.1.11]{Rel-hom}.
Then, trivially, these equivalent assertions imply the assertions $(6)$ and $(7)$.\\
 The equivalences $(1) \Leftrightarrow (4) \Leftrightarrow
(5)$ are the same  as \cite[Theorems 12.3.1  $(1) \Leftrightarrow
(2) \Leftrightarrow (3)$]{Rel-hom}. Then, easily we show that
these
equivalent assertions imply the assertions $(8)$ and $(9)$.\\
For the implications $(6) \Rightarrow (4)$ and $ (7)\Rightarrow
(5)$, use \cite[Theorems 2.1 and 2.2]{HH1}.\\
To prove the implication $(8) \Rightarrow (5)$, consider a
projective $R$-module $M$. Then, it is $(n,0)$-SG-projective for
every $n\geq 1$, and so, by hypothesis,  $M$ is
$(n,m)$-SG-injective. This means that  $M$  is a quotient of an
$R$-module $I$ with injective dimension at most $m$. Then, as a
projective $R$-module, $M$ is a direct summand of $I$. Therefore,
$\id_R(M)\leq m$.\\
Similarly we prove  the implication   $(9) \Rightarrow
(6)$.\end{proof}

Note that \cite[Proposition 2.6]{BM3} shows that we do not need,
in Theorem \ref{thm-Noe}, to assume first that the ring is
Noetherian when $m=0$. In this case the ring $R$ is
quasi-Frobenius (i.e., $0$-Gorenstein).

Aa a consequence, we have for Noetherian $(n,m)$-SG rings the
following result:

\begin{corollary}\label{cor-nm-Noe} Let $n\geq 1$ and $m\geq 0$ be
integers. For  a left and right Noetherian  ring  $R$,  the
following assertions are equivalent:
\begin{enumerate}
    \item   $R$ is left $(n,m)$-SG rings;
    \item   $R$ is right $(n,m)$-SG rings;
    \item   $R$ is $m$-Gorenstein  and  every G-projective (left or right) $R$-module is
    $(n,0)$-SG-projective
    \item     $R$ is $m$-Gorenstein  and  every G-injective (left or right) $R$-module is
    $(n,0)$-SG-injective.
\end{enumerate}
\end{corollary}
\begin{proof}   Apply Theorems \ref{thm-main} and
\ref{thm-Noe}.\end{proof}

In \cite[Section 3]{BMO}, $(1,0)$-SG commutative rings are studied
and they are called SG-semisimple. It is proved that a local
$(1,0)$-SG commutative ring is just a ring with only one
non-trivial ideal \cite[Theorem 3.7]{BMO}. For $(n,0)$-SG rings we
have:

\begin{corollary}\label{cor-n0-Noe} Let $n\geq 1$  be
an integer. For   a ring  $R$,  the following assertions are
equivalent:
\begin{enumerate}
    \item   $R$ is left $(n,0)$-SG rings;
    \item   $R$ is right $(n,0)$-SG rings;
    \item   $R$ is  quasi-Frobenius and  every G-projective (left or right) $R$-module is
    $(n,0)$-SG-projective
    \item     $R$ is quasi-Frobenius and  every G-injective (left or right) $R$-module is
    $(n,0)$-SG-injective.
\end{enumerate}
\end{corollary}
\begin{proof}   Apply \cite[Proposition 2.6]{BM3} and Corollary \ref{cor-nm-Noe}.\end{proof}

Also as a consequence of the main result, we get the following
result which study the relation between rings of finite left
global dimension and left $(n,m)$-SG rings.

\begin{corollary} \label{cor-gldim}   Let $R$ be a ring  and let
$m\geq 0$ be an integer. Then, $\lgldim(R)\leq m$ if and only if
$R$ is  left $(n,m)$-SG for every integer  $n\geq 1$  and
$\wdim(R)<\infty$.
\end{corollary}
\begin{proof} Follows easily from Theorem \ref{thm-main} and \cite[Corollary 1.2]{BM3}.\end{proof}

The following establish some relations between left $(n,m)$-SG
rings.

\begin{proposition}\label{prop-relations} For two integers  $n\geq 1$ and $m\geq 0$, we have the following
assertions:
\begin{enumerate}
    \item Every left $(n,m)$-SG ring is
    $(n,m')$-SG for every $m'\geq m$.
    \item  Every  left $(n,m)$-SG ring  is left
    $(nk,m )$-SG for every $k\geq 1$.\\
    In particular, every  left $(1,m)$-SG ring is
   left $(n,m)$-SG   for every  $n\geq 1$.
\end{enumerate}
\end{proposition}
\begin{proof}   A simple consequence of \cite[Proposition
2.2]{B5}.\end{proof}

Naturally, one would like to have examples of $(n,m)$-SG rings
which are neither  $(n-1,m)$-SG nor  $(n,m-1)$-SG  for every
integers  $n\geq 1$ and $m\geq 0$. In what follows, we give
examples for  $n=1 $ and $2$. For that, we need some change
of ring results.\\

The next result studies  the notion of left $(n,m)$-SG rings in
direct products of rings. For the convenience of the reader, we
recall some properties concerning the structure of modules and
homomorphisms over direct products of rings (for more details
please see \cite[Section 2.6]{BK}).\bigskip

Let $R=\displaystyle\prod_{i=1}^n R_i$ be a direct product of
rings. If $M_i$ is a left (resp., right) $R_i$-module for
$i=1,...,n$, then $M=M_1\oplus\cdots\oplus M_n$ is a left (resp.,
right) $R$-module. Conversely, if $M$ is  a left (resp., right)
$R$-module, then it is of the form $M=M_1\oplus\cdots\oplus M_n$,
where  $M_i$ is  a left (resp., right) $R_i$-module for
$i=1,...,n$ \cite[Subsection 2.6.6]{BK}. Also, the homomorphisms
of $R$-modules are determined by their actions on the $R_i$-module
components.  This is summarized in the following result:

\begin{lemma}[\cite{BK}, Theorem 2.6.8]\label{lem-Hom-structure-product} Let
$R=\displaystyle\prod_{i=1}^n R_i$ be a direct product of rings
and let $M=M_1 \oplus \cdots\oplus M_n$ and
$N=N_1\oplus\cdots\oplus N_n$ be decompositions of  left (resp.,
right) $R$-modules into left (resp., right)  $R_i$-modules. Then,
the following hold:
\begin{enumerate}
    \item There is a natural isomorphism of abelian groups:
 $$\begin{array}{ccl}
   \Hom_R(M,N)&\stackrel{\cong}\longrightarrow   &  \Hom_{R_1}(M_1,N_1)\oplus
   \cdots
    \oplus\Hom_{R_n}(M_n,N_n)  \\
    \alpha & \longmapsto & \alpha_1\oplus \cdots
    \oplus \alpha_n \\
 \end{array}
$$
where the homomorphism $\alpha_1\oplus \cdots
    \oplus \alpha_n$ is defined by: $$(\alpha_1\oplus \cdots
    \oplus \alpha_n)(m_1,...,m_n)=(\alpha_1m_1, ...,
   \alpha_n m_n).$$
    \item The homomorphism $\alpha$ is injective (resp.,
    surjective) if  and only if  each $\alpha_i$ is injective (resp.,
    surjective).
\end{enumerate}
\end{lemma}

Using this result with \cite[Corollary 2.6.9]{BK}, we get the
following known result:

\begin{lemma}\label{lem-prod2} Let
$R=\displaystyle\prod_{i=1}^n R_i$ be a direct product of rings
and let $M=M_1 \oplus \cdots\oplus M_n$ be an $R$-module. Then,
$$\pd_R(M)=\sup\{\pd_{R_i}(M_i);\ 1\leq i\leq n\}.$$
Consequently, $$\lgldim(R)=\sup\{\lgldim(R_i);\ 1\leq i\leq n\}.$$
\end{lemma}

Then, using these results, we get the following result:

\begin{proposition}\label{prop-prod}  Let
$R=\displaystyle\prod_{i=1}^n R_i$ be a direct product of rings
and let $M=M_1 \oplus \cdots\oplus M_n$  be an $R$-module. Then,
for  two integers $n\geq 1$ and $m\geq 0$, $M$ is an
$(n,m)$-SG-projective $R$-module if and only if each $M_i$ is an
$(n,m)$-SG-projective
$R_i$-module.\\
Consequently,  $R$ is left $(n,m)$-SG  if and only if each $R_i$
is left $(n,m)$-SG.
\end{proposition}

Now, we can give the first example.

\begin{example}\label{exm-1}
Consider the the quotient ring $R=\Z/4\Z$, where   $\Z$ denotes
the ring of integers, and consider a family of rings $S_i$ for
$i\geq 1$, such that $\lgldim(S_i)=i$. Then, for every $i\geq 1$,
the direct product of rings $R\times S_i$ is left $(1,i)$-SG, but
it is not left $(1,i-1)$-SG.
\end{example}
\begin{proof}  From \cite[Corollary 3.9]{BMO}, $R$ is a commutative
$(1,0)$-SG ring. Then, using Proposition \ref{prop-prod} and Lemma
\ref{lem-prod2}, we get, for every $i\geq 1$,   that the ring
$R\times S_i$ is left $(1,i)$-SG, but it is not left
$(1,i-1)$-SG.\end{proof}

We end with examples of $(2,m)$-SG rings.

\begin{proposition}\label{prop-2m}
Let $R$ be a commutative ring with $\gldim(R )=m$   for some
integer $m\geq 1$, such that $R$ contains a non-zero divisor
element $x$. Then, the quotient ring $ R /xR $ is   $(2,m-1)$-SG.
\end{proposition}
\begin{proof}  Let $M$ be an $ R /xR $-module. Then, using the canonical
surjection of rings $R\rightarrow R /xR $, $M$ is an $R$-module.
Thus, there exists an exact sequence of $R$-modules,
$$0\rightarrow Q\rightarrow P\rightarrow M\rightarrow 0,$$ where $P$ is projective and  $\pd_R(Q)\leq
m-1$. Tensoring this sequence by $R /xR $, we get:
$$ \Tor_1^R(R/xR,P)=0  \rightarrow\Tor_1^R(R/xR,M) \rightarrow Q/xQ\rightarrow P/x P\rightarrow M \rightarrow 0.$$
From \cite[Examples (1), p. 102]{Bou}, $\Tor_1^R(R/xR,M)=M$ since
$xM=0$, and so we get an exact  of $R/xR$-modules of the form:
$$ 0  \rightarrow M \rightarrow Q/xQ\rightarrow P/x P\rightarrow M \rightarrow 0.$$
The $R/xR$-module  $P/x P$ is projective, and since $Q$ is a
submodule of the projective $R$-module $P$, $x$ is also a non-zero
divisor element on $Q$, then $\pd_{R/xR}(Q/xQ)\leq\pd_R(Q)\leq
m-1$ (from \cite[Theorem 4.3.5]{Wei}). Then, to see that $M$ is
$(2,m-1)$-SG-projective, it remains to show that
$\Ext_{R/xR}^{i}(M,F)$ for every projective $R/xR$-module $F$ and
for every $i\geq m$. For that, it suffices to consider $F$ to be a
free $R/xR$-module. In that case $F$ is of the form $L/xL$, where
$L$ is a free $R$-module. From Rees's theorem and since  $\gldim(R
)=m$, we have for every $i\geq m$:
$$  \Ext_{R/xR}^{i}(M,F ) \cong  \Ext_{R/xR}^{i}(M,L/xL) \cong \Ext_{R}^{i+1}(M,L)=0.$$
This completes the proof.\end{proof}

\begin{example}\label{exm-2m}
Consider the the quotient ring $R=\Z/8\Z$, where   $\Z$ denotes
the ring of integers, and consider a family of rings $S_i$, for
$i\geq 0$, such that $\lgldim(S_i)=i$. Then:
\begin{itemize}
    \item the direct product of
rings $R\times S_0$ is left $(2,0)$-SG which is not left
$(1,0)$-SG, and
    \item for $i\geq 1$, the direct product of
rings  $R\times S_i$ is left $(2,i)$-SG, but it is neither  left
$(1,i )$-SG nor  $(2,i-1)$-SG.
\end{itemize}

\end{example}
\begin{proof}  From Proposition \ref{prop-2m},  $R$ is $(2,0)$-SG, and,
by Lemma \cite[Corollary 3.10]{BMO}, $R$ is not $(1,0)$-SG. Then,
from Proposition \ref{prop-prod},  $R\times S_0$ is left
$(2,0)$-SG, but it is not left $(1,0)$-SG.\\
Using Lemma \ref{lem-prod2}, the same argument as above shows, for
$i\geq 1$, that $R\times S_i$ is left $(2,i)$-SG, but it is
neither left $(1,i )$-SG nor  $(2,i-1)$-SG.\end{proof}



\end{document}